\documentclass[AMA,STIX1COL]{WileyNJD-v2}

\articletype{Article Type}%

\received{26 April 2016}
\revised{6 June 2016}
\accepted{6 June 2016}

\newcommand{\R}{\mathbb{R}}
\newcommand{\C}{\mathbb{C}}
\newcommand{\N}{\mathbb{N}}

\begin{document}
\title{Robustness of Exponential Stability of a Class of
	Switched Linear Systems with State Delays\protect\thanks{This is an example for title footnote.}}

\author[1]{Nguyen Khoa Son*}

\author[2]{Le Van Ngoc}
\authormark{N.K.Son \& L.V.Ngoc }

\address[1]{Institute of Mathematics,
	Vietnam Academy of Science and Technology, 18 Hoang Quoc Viet Rd., Hanoi, Vietnam: nkson@vast.vn}%
\address[2]{Department of Scientific Fundamentals, Posts and Telecommunications Institute of Technology, Km10 Nguyen Trai Rd, Hanoi, Vietnam:ngoclv@ptit.edu.vn or lvngocmath@gmail.com}

\corres{*Nguyen Khoa Son, Institute of Mathematics,
	Vietnam Academy of Science and Technology, 18 Hoang Quoc Viet Rd., Hanoi, Vietnam. \email{nkson@vast.vn}}


\abstract[Summary]{
	This paper investigates the robustness of exponential stability of a class of switched systems described by linear functional differential equations under arbitrary switching. We will measure the stability robustness of such a system, subject to parameter affine perturbations of its constituent subsystems matrices, by introducing the notion of structured stability radius. The lower bounds and the upper bounds for this radius are established under the assumption that certain associated positive linear systems upper bounding the given subsystems have a common linear copositive Lyapunov function. In the case of switched  positive linear systems with discrete multiple delays or distributed delay the obtained results yield some tractably computable bounds for the stability radius. Examples are given to illustrate the proposed method.}


\maketitle


\section{INTRODUCTION}\label{sec1}
The {\it stability robustness} of dynamical systems subject to parameter perturbations or uncertainties has attracted significant attention from  researchers for many years. The motivation comes from  engineering practice. In order to study or control a real system with complicated dynamic behavior the engineer usually considers a simplified mathematical model which is called a nominal system. Then there arises the question of whether a desired property established for the nominal system, say asymptotic stability or controllability, is robust enough to be true when applied to the real system. Since a mathematical model never exactly represents the dynamics of a physical system, the robustness issue is not only important in the context of model reduction but is a fundamental problem for the application of systems and control theory in general.
 
 One of the most effective approaches in measuring the system stability robustness is based on the concept of {\it stability radius} which was introduced and analyzed in the seminal  works \cite{ H-P2, Qiu} for  the linear systems in $\R^n$ of the form
$\dot x(t)=A_0 x(t), \ t\geq 0,	$
which represents the unperturbed asymptotically stable nominal system (i.e. the zeros of  the characteristic polynomial $\det (sI-A_0)=0$ are contained in $\C_- $ -the open left-half complex plane). The stability radius of the system is defined as the \textit{ maximal number} $\delta_0>0$ such that all the perturbed systems $\dot x(t)= \widetilde A_0 x(t)= (A_0+D\Delta E)x(t)$ are asymptotically stable whenever $\|\Delta\| <\delta_0$ (with some matrix norm $\|\cdot\|$)  where $D$ and $E$ are given real  matrices of appropriate sizes and $\Delta$ is an unknown perturbation matrix. Depending on perturbation matrix $\Delta$  being real or complex one gets, correspondingly, the real or complex stability radii, which are basically different \cite{H-P-differences}. Geometrically, when $D$ and $E$ are the identity matrix, the stability radius is just the distance to instability of an asymptotically stable nominal system and the calculation of this radius is reduced to a nonconvex optimization problem (since the set of unstable systems is not convex).  Problems of characterizing and calculating  stability radii for different classes of dynamical systems and subject to different types of perturbations or uncertainties have been a topic of major interest in  the system and control theory over the past several decades. The interested reader is referred to the survey article \cite{H-P-survey} and the monograph \cite{H-P} where an extensive literature on this study can be found; see also \cite{Hin_Son_IJNRC98} for discrete-time systems and \cite{Son-Hinrichsen96} for \textit{positive systems}, where it has been show that, for this class of systems, the real and the complex stability radii coincide and can be calculated by a simple formula. Note that a  dynamical system with state space $\R^n$ is called positive if any trajectory of the system starting at an initial state in the positive orthant $\R^n_ +$  remains forever in $\R^n_+$. This class of systems is used in many areas such as economics, populations dynamics and ecology, see, e.g. \cite{Farina}. The mathematical
theory of positive systems is based on the theory of nonnegative matrices \cite{Berman}, where the famous Perron-Frobenius Theorem plays an essential role, particularly in the system stability analysis. 

 Further, the stability radius problem has been considered intensively also for time-delay systems where the system's equation depends not only on the present but also on the past state. For instance, under the assumption that the linear system with discrete delays of the form
$ \dot x(t)= A_0 x(t) + \sum_{i=1}^m A_ix(t-h_i),\ t\geq 0, \ h_i>0$ 
 is asymptotically stable (or, equivalently, the zeros of the characteristic quasi-polynomial $\det(sI-A_0-\sum_{i=1}^mA_ie^{-sh_i})=0$ are contained in $\C_-$, see e.g. \cite{Hale}) and subjected to the {\it structured affine perturbations} $A_i \rightarrow \widetilde A_i= A_i+D_i\Delta_iE_i, \ i=0,1,\ldots, m$ where  $D_i,E_i, i=0,1,\ldots,m$ are given structuring matrices and $\Delta_i, i=0,1,\ldots, m$ are unknown perturbations, a number of results on analysis and computation of the system's stability radius was obtained  in \cite{Son_Ngoc_Acta,Hu_Davison,Ngoc_Son_NFAO2004,Michiels_Friedman2009, Michiels_Niculescu2014,Borgioli2019}. Some extensions have been done in \cite{Son_Ngoc2001,Ngoc_Son_SIAM} for a more general class of systems described by linear functional differential equations (or FDEs, for short) of the form
\begin{equation}\label{delay2}
\dot x(t)= A_0x(t) + Lx_t, \ t\geq 0,	
	\end{equation}
where $x_t(\theta):= x(t+\theta)$ is a continuous function of $\theta\in [-h,0]$  and $L: C([-h,0],\R^n)\rightarrow \R^n$ is a linear bounded operator, which covers, as particular cases, the aforementioned linear systems with discrete delays and those with the distributed delays. Also, it has been shown that for the classes of positive delay linear systems the system's real stability radius can be calculated directly via a simple formula expressed in terms of the systems's matrices, see e.g. \cite{Son-Hinrichsen96, Hin_Son_IJNRC98,Son_Ngoc2001, anh_son_RNC2009}. 

Given the mentioned widespread popularity of stability radii in the robustness stability analysis, it is surprising that this approach has not been developed so far in the literature on {\it stability of switched systems}.  We recall that a switched system is a type of hybrid dynamic systems which consists of a family of subsystems and a rule called a switching signal that chooses an active subsystem from the family at every instant of time. In the simplest continuous-time model, a linear switched system can be represented in the form
\begin{equation} \label{SwLS}
\dot{x}(t)=A_{\sigma(t)}x(t),\  t\geq 0,\  \sigma \in \Sigma_+,
\end{equation} 
where $\Sigma_+$ is a set of piece-wise constant functions $\sigma: [0,+\infty) \rightarrow \underline N:=\{1,2,\ldots,N\}$ (satisfying some realistic assumptions) which defines a switching between given constituent subsystems 
$	\dot{x}(t)=A_kx(t),\   t\geq 0,\ k\in \underline N.$
The study on switched systems has drawn considerable  attention in system and control community  over the past few decades due to  wide range of applications of this type of systems in engineering practice. The reader is referred to monographs \cite{Liberzon}, \cite{Sun},  survey papers \cite{Lin,shorten2007} and the references therein for more details on different problems regarding switched systems, in particular, on stability issues. It has been indicated, for instance,   that  the  switched linear system \eqref{SwLS} is exponentially stable {\it under arbitrary switching} $\sigma$   if all constituent subsystems have a common  quadratic Lyapunov function (QLF). The last condition, however, is not necessary for exponential stability of \eqref{SwLS} and many efforts have been made to reduce  this conservatism, by looking for more general types of Lyapunov functions. Recently, similar problems have been  considered intensively also for time-delay switched systems, where different kinds of the so-called Lyapunov - Krasovskii functionals are  playing a similar role  (see, e.g. \cite{Kim, Wu, Meng2017}).  In the meantime, for the class of  positive or compartmental  switched linear systems, besides traditional quadratic Lyapunov functions, a more restrictive notion of linear copositive Lyapunov functions (LCLF) is exploited effectively in the study of stability problems (see e.g.  \cite{ Chelaboina2004, Mason2007, Knorn_Mason2009, Fornasini,   Blanchini2015, Valcher2016} and  also \cite{Liu_Yu_Wang, Liu_Dang,   YSun2016, Li2018, Liu2018} for  time-delay systems and the comparison method).

We would like to emphasize that the above references are all dedicated to the stability analysis of switched systems under arbitrary switching signals. Although in engineering practice, switching laws must usually satisfy some technical restriction (for instance, the time between switching instances must be long enough) such a strong requirement is important when the switching mechanism is unknown, or too complicated to be useful in the stability analysis. Moreover, as mentioned in \cite{Liberzon}, modern computer-controlled systems are capable of very fast switching rates, which creates
the need to be able to test the stability of the switched system for arbitrarily fast switching signals. On the other side, while it is true that the exponential stability problem under arbitrary switching has received most of the recent attention in the switched systems literature and will also be a subject of this paper, a number of other stability problems stand out as being worthy of attention (see, e.g. \cite{Lin,shorten2007} and \cite{Lib_Mor}). In particular, the problem of stability and stabilization of switched systems, using the so-called average dwell time switchings (or ADT, for short) has attracted a considerable attention, see e.g. \cite{Liberzon,  hespanha, geromel, zhai, Wick2} and \cite{Zhang2018, Wang2019, Yu2021}, for more recent contributions.  Another problem of interest in this context is that of determining stabilizing switching rules for systems with partly or all unstable constituent subsystems \cite{Yang2018, Li18, Mao2019, Lu2020}. However, these problems are not in the scope our consideration in this paper.

Turning back to the problem to be studied in this paper, consider the linear switched system \eqref{SwLS} which is assumed exponentially stable under arbitrary switching.  Then, based on the stability radius approach mentioned above, in order to measure the system's stability robustness, we can put the problem of computing the smallest size $\min_{k\in \underline N}\|\Delta_k\|$ of unknown disturbances $\Delta_k $  such that the perturbed  switched system $\dot x(t)= \widetilde A_{\sigma(t)}x(t)$ (with $\widetilde A_k= A_k+D_k\Delta_kE_k, \ k\in \underline N$ and $ D_k, E_k, k\in \underline N$ being given matrices defining the structure of perturbations) is  not asymptotically stable, under some switching signal, and define this quantity as the system's stability radius. Such a problem was recently considered in \cite{Son_Ngoc_IET}, for the first time in the literature, to the best of our knowledge, although its significance has been mentioned much earlier, e.g. in \cite{shorten2007}. We were able to established some computable lower bounds and upper bounds for the stability radius of the {\it non-delay} switched linear system \eqref{SwLS},  under the assumption that subsystems have a common QLF  or LCLF (if $A_k$ are Metzler matrices). 

In this paper we will extend the approach which was developed in \cite{Son_Ngoc_IET} for non-delay switched systems to study the robustness of exponential stability for  the class  of {\it time-delay} switched systems described by linear functional differential equations (FDE, for short) of the form
\begin{equation}	\label{DSwFDSa}\dot{x}(t)=A^0_{\sigma(t)}x(t)+ L_{\sigma(t)}x_t, \   t\geq 0, \ \sigma  \in \Sigma_+, \end{equation}
under the assumption  that the nominal constituent subsystems 
$	\dot x(t)=A^0_kx(t)+L_k x_t,\  t\geq 0, k\in \underline N$
(or, equivalently, matrices $A_k^0$ and linear operators $L_k, k\in \underline N $)  are subjected to structured affine perturbations. The primary purpose is to establish the computable formula for estimating the stability radius of this class of systems. Although there is a resemblance in the statement of some results in this paper to those of  \cite{Son_Ngoc_IET}, the proofs are basically different, requiring more sophisticated mathematical tools in infinite-dimensional spaces. In particular, the  results from the stability theory of FDEs \cite{Hale, Ngoc_Naito} and those on the robust stability of positive FDEs \cite{Son_Ngoc2001, Ngoc_Son_SIAM} are needed in the proofs. As an advantage of this general consideration, we obtain a unified framework for the studying the robustness of stability of  switched linear time-delay systems, including those with both discrete multiple delays and distributed delays as particular cases.

The remainder of the paper is organized as follows. In Section 2 we present the notation and  mathematical background necessary to state the main results of
the paper. In Section 3 we present some criteria for exponential stability of switched systems described by linear FDEs, including those for positive systems, which will be used in robustness analysis. In Section 4, for analyzing the robustness of stability, a definition of stability radius of  time-delay switched linear systems is given and some formulas for computing its bounds are established. Two examples are  given to illustrate the obtained results. Finally, in Conclusion we summarize the main contribution of the paper and give some remarks on future work.\\

\section{PRELIMINARIES}\label{sec2}
In this section, we introduce the main notation and present a number of preliminary results to be used in what follows. For a positive integer $ r, \underline r$ denotes the set of numbers $ \{1,2,\ldots, r \}$. Throughout, $\R$ denotes the field of real numbers,  $\R^n$ is the
vector space of all $n$-tuples of real numbers and $\R^{n\times m}$ is the space
of $(n\times m)$-matrices $(a_{ij})$ with entries $ a_{ij}\in \R$.  For  matrices   $A=(a_{ij})$ and $B=(b_{ij})$ in $\mathbb{R}^{n\times m}$, we write $A\geq B$ and $A \gg B$ iff $a_{ij} \geq b_{ij}$ and $a_{ij}>b_{ij}$ for $i\in \underline n,j\in \underline m, $ respectively. $|A|$ stands for the matrix $(|a_{ij}|)$ and $A^{\top}$ is the transpose of $A$. Similar notation is applied for vectors $x\in \R^n$. Without lost of generality, we assume that $\R^n$ is equipped with the $\infty$-norm: $\|x\|=\|x\|_{\infty}=\max_{1\leq i\leq n}|x_i|$ and the norm in $\R^{n\times m} $ is the induced operator norm: $ \|A\|=\max_{\|x\|=1}\|Ax\|= \max_{1\leq i\leq n}\sum_{j=1}^m|a_{ij}|.$
 The maximal real part of any eigenvalue of $A\in \R^{n\times n} $ is denoted by $\mu(A)$. If $\mu(A) < 0 $ (i.e. all the eigenvalues of $A$ are in $\C_-$ -the open left-half complex
plane) then $A$ is said to be {\it Hurwitz stable}. 

Further, for the convenience of the readers, let us recall some well-known facts from functional analysis  which are used frequently in the theory of FDEs (see, e.g.\cite{Hale, Bachman} for details). For $h>0,  C([-h, 0], \R^n)$ denotes the Banach space of continuous functions $\varphi: [-h, 0] \rightarrow \R^n $  with the norm $\|\varphi\|= \max_{\theta\in[-h,0]} \|\varphi(\theta)\|$; $NBV([-h,0],\R)$ will stand for the linear space of  all normalized  functions of bounded variation $\psi: [-h,0] \rightarrow \R$, which are left-side continuous on the interval $(-h,0], \psi (-h) = 0$ and have the bounded total variation $ Var([-h,0],\psi) =  \sup_{P[-h,0]}\sum_{k}|\psi(\theta_k)- \psi(\theta_{k-1})| <\infty $, 
the supremum being taken over the set of all finite partitions of the interval $[-h,0]$. Then it is well-known (see, e.g. Chapter 12 of\cite{Bachman}) that $NBV([-h,0],\R)$ provided with the norm $\|\psi\|= Var([-h,0],\psi)$ is a Banach space which is isometrically isomorphic to $C^*([-h,0],\R)$ -the dual space of $C([-h,0],\R)$ (the Riesz Representation Theorem). In other words, for any linear bounded functional $f: C([-h,0],\R)\rightarrow \R$, there exists a unique $\psi\in NBV([-h,0],\R)$ such that
\begin{equation*}\label{Riesz}
	f(\varphi)=\int_{-h}^0d[\psi(\theta)]\varphi(\theta)\ \ \text{and}\ \  \|f\|=\|\psi\|, \ \forall \varphi\in C([-h,0],\R),
	\end{equation*}
where the integral is understood in the sense of Riemann-Stieltjes. It is worthy of mention that  the requirement of $\psi$ being normalized is essential for proving its uniqueness in the integral representation of $f$ as well as the equality of norms in this representation, see e.g. \cite{Bachman}. In this paper we need, actually, the following multi-dimensional version of the  Riesz theorem, see e.g. \cite{Warga} (Theorem I.5.9 ). Denote by $NBV([-h,0],\R^{p\times q})$  the Banach space of all matrix functions $\delta :[-h,0] \rightarrow \R^{p\times q}$ such that  $\delta_{ij}(\cdot) \in NBV([-h,0],\R), \forall i\in \underline p, \forall j \in \underline q,$  equipped with the norm 
\begin{equation}
	\label{normNBV}
	\|\delta\|= \max_{1\leq i\leq p}\sum_{j=1}^q Var([-h,0],\delta_{ij}). 
\end{equation} If  $L: C([-h,0],\R^q) \rightarrow \R^p$ is a linear bounded operator, then by the Riesz Representation Theorem, there is a unique $\delta \in NBV([-h,0],\R^{p\times q})$ such that
\begin{equation*}\label{L}
L(\varphi)= \int_{-h}^0d[\delta(\theta)]\varphi(\theta), \ \forall \varphi \in C([-h,0],\R^q),
\end{equation*} 
where, by definition, 
\begin{equation*}\label{Li}
\big(L(\varphi)\big)_i= \sum_{j=1}^q\int_{-h}^0d[\delta_{ij}(\theta)]\varphi_j(\theta),\ \forall \varphi= \big(\varphi_1,\varphi_2,\ldots,\varphi_q)^{\top}, \ \forall i\in \underline p.	
	\end{equation*} 
Thus, to each $\delta \in NBV([-h,0],\R^{p\times q}) $ we can associate a nonnegative $(p\times q)$-matrix 
\begin{equation}
\label{var}
V(\delta):= \big(Var([-h,0],\delta_{ij})\big) \geq 0.
\end{equation}
It follows from the definition that  if  both $\R^p$ and $\R^q$  are provided with the norm $\|\cdot\|_{\infty}$ then, for any $\Delta\in \R^{p\times q}, \delta \in NBV([-h,0],\R^{p\times q})$  and any constant matrices $D\in \R^{n\times p}, E\in \R^{q\times  n}$, we have
\begin{equation}
\label{ineqvar}
\|\  |\Delta|\ \|= \|\Delta\|,\ \ \|V(\delta)\| = \|\delta\|, 
\end{equation}
\vskip-0.5cm
\begin{equation}
\label{ineqvar1}
D\delta E\in NBV([-h,0],\R^{n\times n})\ \ \ \text{and}\ \ \ \|D\delta E\|\leq \|D\|\|\delta\|\|E\|.
\end{equation}

Finally, recall that $A\in \R^{n\times n}$ is said to be a {\it  Metzler matrix} if all  off-diagonal elements
of $A$ are nonnegative: $a_{ij}\geq 0,$ if $ i\not= j$. For an arbitrary matrix $A\in \R^{n\times n}$ we can associate the  Metzler matrix $\mathcal{M}(A)$, by defining, 
\begin{equation}\label{Metzler}
\mathcal{M}(A) := (\hat a_{ij}), 
\ \hat a_{ii} = a_{ii},\forall i\in \underline n,\ \text{and} \   \ \hat a_{ij} = |a_{ij}|, \forall i\not=j\in \underline n.
\end{equation} 
It can be easily verified that
\begin{equation}\label{M}
	\mathcal{M}(A+B) \leq \mathcal{M} (A)+ |B|, \ \forall A,B\in \R^{n\times n}. 
\end{equation}
Some well-known properties of Metzler matrices are collected in the following lemma (see. e.g. \cite{HornJohnson_91}) which will be used in the sequel.
\begin{lemma}\label{lemma2.1} 
	{\it Let $A\in \mathbb{R}^{n\times n}$ be a Metzler matrix. Then the following statements are equivalent: ${\rm(i)} \; A\ \text{is Hurwitz, i.e.}\\ \mu(A)<0;\ {\rm(ii)}\ Ap\ll 0$ for some $p \in \mathbb{R}_{+}^{n}, p\gg 0$;\  {\rm(iii)}  $A$ is invertible and $- A^{-1}\geq 0$.
	}
\end{lemma}
Moreover, for any $A\in \R^{n\times n}$, we have (see, e.g. \cite{Son-Hinrichsen96}),
\begin{equation}\label{muA}
\mu (A) \leq \mu(\mathcal {M}(A)).
\end{equation} Therefore, $A$ is Hurwitz if $\mathcal {M}(A)$ is Hurwitz but not conversely, in general.

\section{CRITERIA FOR EXPONENTIAL STABILITY}\label{sec2}

In this section, we will present some criteria for checking exponential stability of linear switched systems described by linear FDEs that will be used in the next section to estimate the stability robustness of  systems.\\

Consider a  switched linear system whose dynamics are described by the FDE of the form \eqref{delay2}. Then, by Riesz Representation Theorem,  as we recalled in Section 2, this system can be equivalently represented in the integral form
\begin{equation}
\label{DSwFDS}
\dot{x}(t)=A^0_{\sigma(t)}x(t)+\int_{-h}^0d[\eta_{\sigma(t)}(\theta)]x(t+\theta),  t\geq 0, \sigma  \in \Sigma_+,
\end{equation}
where for each $t\geq 0, A^0_{\sigma(t)}\in \mathcal{A} := \{A_k^0, k\in \underline{N}\} \subset \R^{n\times n}, $ -a given family of $N$  real matrices and $ \eta_{\sigma(t)} \in  \Gamma:= \{\eta_k, k\in \underline{N}\} \subset NBV([-h,0],\R^{n\times n})$, -a given family of $N$  matrix functions with normalized bounded variation elements $\eta_{k,ij}(\cdot)$. In what follows $\Sigma_+$ is assume to be the set of all switching signals  $\sigma : [0,+\infty ) \rightarrow \underline N $ which  are piece-wise constant, right-side continuous  functions, with points of discontinuity $\tau_k, k=1,2,\ldots $ (known as the {\it switching
	instances}) having a strictly positive {\it dwell time}:  $ \tau(\sigma) := \inf_{k\in \N} (\tau_{k+1}-\tau_k) >0 $.  Thus, each $\sigma$ performs  switchings  between the following $N$ time-delay linear subsystems $(A^0_k,\eta_k): $ 
\begin{equation}
\label{DSwFDSk}
(A^0_k,\eta_k): \hskip3.5cm \dot{x}(t)=A^0_kx(t)+\int_{-h}^0d[\eta_k(\theta)]x(t+\theta), \  t\geq 0,	\ k\in \underline N, \hskip4.0cm
\end{equation}
where the $i$-th component of the second term in \eqref{DSwFDSk}, for each  
$k\in \underline N$ and $ i=1,\ldots,n, $ is defined as
\begin{equation*}\label{stieljest}
\bigg(\int_{-h}^{0}d[\eta_k(\theta)]x(t+\theta)\bigg)_i=\sum_{j=1}^{n}\int_{-h}^{0}d[\eta_{k,ij}(\theta)]x_j(t+\theta),  
\end{equation*}
the integrals in the right-side above being understood in the Riemann-Stieltjes sense. In what follows we will sometimes reffer to the swiched system \eqref{DSwFDS} with the constituent subsystems \eqref{DSwFDSk} as the system \eqref{DSwFDS}-\eqref{DSwFDSk} when the link between them needs to be indicated. 

For any $\varphi\in C([-h,0],\R^n) $ and any switching signal $\sigma \in \Sigma_+$, the system \eqref{DSwFDS} admits a unique solution  $x(t)=x(t,\varphi, \sigma),$ $ t\geq -h, $ satisfying the initial condition $x(\theta)=\varphi(\theta),\; \theta\in[-h,0].$ Note that the solution $x(t) $ is absolutely continuous function on $[0,+\infty)$ and differentiable everywhere, except for the set of the switching instances $\{\tau_k\}$ of $\sigma $  where $x(t)$ has only Dini right- and left-derivatives $D^+x(\tau_k), D^-x(\tau_k) $  which are generally different.

The system \eqref{DSwFDS}-\eqref{DSwFDSk} is said to be {\it positive} if   $x(t)\geq 0,  \forall t\geq 0 $ whenever  $ \varphi(\theta) \geq 0, \ \forall \theta \in [-h,0]$. It is trivial  to show that the system is positive if and only if all subsystems  \eqref{DSwFDSk} are positive. The latter  is equivalent to the condition that, for each $k\in \underline N, A_k^0$ is a Metzler matrix and $\eta_k$ is increasing on $[-h,0]: 0=\eta_k(-h)\leq \eta_k(\theta_1)\leq \eta_k(\theta_2)$, if $-h\leq \theta_1 \leq\theta_2\leq 0$ (see, e.g. \cite{Son_Ngoc2001, Ngoc_Naito}). 

\begin{definition} 
	\label{defDSwFDS}
	The switched system \eqref{DSwFDS}-\eqref{DSwFDSk} is said to be  exponentially stable under arbitrary switching if there exist real numbers $M >0, \alpha >0 $ such that for any $\varphi\in C([-h,0],\R^n) $ and any switching signal  $\sigma \in \Sigma_+,$ the solutions $ x(t,\varphi, \sigma) $  of \eqref{DSwFDS} satisfies
	\begin{equation}
	\label{conditiondef1}
	\|x(t,\varphi,\sigma)\|\leq M e^{-\alpha t}\|\varphi\|, \quad \forall t\geq 0.
	\end{equation}
\end{definition}
Clearly, for the switched system \eqref{DSwFDS}-\eqref{DSwFDSk} to be exponentially stable it is necessary that all of the individual subsystems \eqref{DSwFDSk} are exponentially stable or, equivalently (see e.g. \cite{Hale}), all zeros of their characteristic quasi-polynomials $\text  {det} \; P_k(s)$, which are defined as 
\begin{equation} 
\label{quasicharact}
P_k(s)=sI-A^0_k-\int_{-h}^0e^{s\theta}d[\eta_k(\theta)], 
k\in \underline N,
\end{equation}
have negative real parts. Similarly to the case of switched systems with no delays (i.e. when $\eta_k \equiv 0, \forall k$), the last condition is, however,  not sufficient for exponential stability of \eqref{DSwFDS}-\eqref{DSwFDSk} under arbitrary switching.

The following theorem gives a verifiable sufficient condition for exponential stability of the class of delay switched linear systems  of the form \eqref{DSwFDS} that will be used in the next section for analyzing stability robustness. The main idea of the proof is essentially based on the  comparison principle of solutions (see and compare with  \cite{Liu_Dang, NgocAMS2013}). We give the detailed proof for the convenience of the readers. 

\begin{theorem}
	\label{main_lemma} 
	Consider the time-delay  switched linear system  \eqref{DSwFDS}-\eqref{DSwFDSk}.
	If there exists strictly positive vector $ \xi \gg0$ such that	
	\begin{equation}
	\label{cond2}
	\big(\mathcal{M}(A^0_k)+ V(\eta_k)\big)\xi \ll 0, \  \forall k \in \underline N,
	\end{equation}
	where the  matrix of variations  $V(\eta_k)$ and the  Metzler matrix  $ \mathcal{M}(A^0_k) $ are defined, respectively, by \eqref{var}  and \eqref{Metzler}, then the  system  is exponentially stable under arbitrary switching. 
\end{theorem}
\begin{proof}
Due to the  linearity of the system, it suffices to prove \eqref{conditiondef1} for any $\varphi\in C([-h,0],\R^n)$ with $\|\varphi\|\leq 1$, that is
\begin{equation}
\label{estim1}
\|x(t)\| \leq Me^{-\alpha t}, \ \forall t \geq 0,
\end{equation} where  $x(t)=x(t,\varphi,\sigma) $ is the solution of \eqref{DSwFDS} satisfying the initial condition $x(\theta)=\varphi(\theta),\; \theta\in[-h,0]$. To this end, let \eqref{cond2} hold for  some $\xi:=(\xi_1,\xi_2,...,\xi_n)^{\top}\in \R^n$ with  $ \xi_i>0, \forall i\in \underline{n}.$ Then, by continuity, one can choose a sufficiently small $\alpha >0$ such that 
\begin{equation}
\label{metzlerbeta}
(\mathcal{M}(A^0_k)+e^{\alpha h} V({\eta}_k))\xi\ll -\alpha \xi, \  \forall k \in \underline N.
\end{equation}
Choosing a real number $M > \dfrac{\|\xi\|}{ \min_{i\in \underline n}\xi_i }= \dfrac{\max_{i\in \underline n}\xi_i}{ \min_{i\in \underline n}\xi_i }\geq 1$  and setting
\begin{equation}
\label{defu}
y(t) = Me^{-\alpha t}\dfrac{\xi}{\|\xi\|},\ t\geq 0,
\end{equation} 
we have, obviously, 
\begin{equation}\label{d}
	 |x(t)|< y(t), \ \forall t\in [-h,0].
\end{equation}
 By the definition of the $\infty$-norm,  in order to prove \eqref{estim1}  it suffices to verify that 
\begin{equation}
\label{def2}
|x_i(t)|\leq y_i(t), \  \forall t\geq 0,\ \forall i\in \underline n. 
\end{equation} 
Assume to the contrary that \eqref{def2} does not hold for some solution  $x(t)=x(t,\varphi_0,\sigma_0)$ with  $\|\varphi_0\|\leq 1$ and $\  \sigma_0\in \Sigma_+$. Then, there exists $i'\in \underline n$ such that the set $ T_{i'}:=\big\{t>0:  |x_{i'}(t)|> y_{i'}(t)\big\} $ is nonempty. Defining $\bar t_0 \geq 0$ and $i_0\in \underline n$ by setting $\bar t_0=  \min_{i'\in \underline n}\inf \big\{ t\in T_{i'} \big\} = \inf \big\{ t\in T_{i_0} \big\}$, we have, by the continuity and \eqref{d} that $\bar t_0> 0 $ and
\begin{equation}\label{conditionproof 2}
	|x_{i}(t)|\leq y_{i}(t),\  \forall i\in \underline n,\ \ \forall t\in [-h,\bar t_0), \ \text{and}\ \ 	|x_{i_0}(\bar t_0)|= y_{i_0}(\bar t_0),\
\end{equation}
and there exist a sequence $t_k\searrow \bar t_0$ such that
\begin{equation}\label{io}
	|x_{i_0}(t_k)| >y_{i_0}(t_k), k=1,2, \ldots 
\end{equation} 
\noindent Assume $\{ \tau_k\}_{k=1}^{+\infty}$ be the sequence of switching instances of $\sigma_0$ and $\bar t_0\in [\tau_{m_0},\; \tau_{m_0+1})$, then, by the assumption on the dwell time of switching signals, one can choose a sufficiently small $\epsilon_1>0, $ such that $  [\bar t_0,\ \bar t_0+\epsilon_1) \subset [\tau_{m_0},\; \tau_{m_0+1})$. Letting $\sigma_0(t)=\sigma_0(\tau_{m_0}):=k_0, $ for $t\in [\tau_{m_0},\; \tau_{m_0+1})$ then, by definition, the solution $x(t)$ satisfies
\begin{equation*}
\dot{x}(t)=A^0_{k_0}x(t) +\int_{-h}^0d[\eta_{k_0}(\theta)]x(t+\theta),  t\in  [\tau_{m_0},\tau_{m_0+1}).
\end{equation*}
Denoting elements of $A^0_{k_0}$ and $\eta_{k_0}(\theta)$, respectively, by $ a_{k_0, ij}$ and $ \eta_{k_0,ij}(\theta), i,j \in \underline n$ then, from the last equation,  we can deduce easily, for any $t\in [\tau_{m_0},\; \tau_{m_0+1})$ and  every $i\in\underline{n}$, that
\begin{equation}\label{D+}
D^+|x_i(t)|  \leq  a_{k_0,ii}|x_i(t)|+\sum_{j=1,j\neq i}^n|a_{k_0,ij}|\;|x_j(t)| + \sum_{j=1}^n(V({\eta}_{k_0}))_{ij}\max_{\theta\in[-h,0]}|x_j(t+\theta)|.
\end{equation}
Therefore, by  \eqref{metzlerbeta}, \eqref{defu} and \eqref{D+} (with  $i=i_0, t=\bar t_0), $ and using the equality in \eqref{conditionproof 2} we get
\begin{equation*}  
D^+|x_{i_0}(\bar t_0)|\leq \dfrac{M}{\|\xi\|}e^{-\alpha \bar t_0}\big(\big(\mathcal{M}(A^0_{k_0})+e^{\alpha h}V({\eta}_{k_0})\big)\xi\big)_{i_0}<\dfrac{M}{\|\xi\|}e^{-\alpha \bar t_0}(-\alpha\xi_{i_0})=\dfrac{d}{dt}y_{i_0}(\bar t_0).
\end{equation*}
On the other hand, by definition of  the Dini right-derivative and \eqref{conditionproof 2},\eqref{io} we have
\begin{equation*}
	D^+|x_{i_0}(\bar t_0)|= \limsup_{\epsilon\searrow0^+}\dfrac{|x_{i_0}(\bar t_0+\epsilon)|-|x_{i_0}(\bar t_0)|}{\epsilon}\geq \limsup_{k\rightarrow +\infty}\dfrac{|x_{i_0}(t_k)|-|x_{i_0}(\bar t_0)|}{t_k-\bar t_0} \geq\lim_{k\rightarrow +\infty}\dfrac{y_{i_0}(t_k)-y_{i_0}(\bar t_0)}{t_k-\bar t_0}=\dfrac{d}{dt}y_{i_0}(\bar t_0), 
\end{equation*}
contradicting the previous strict inequality. This completes the proof. \end{proof}
Assume that the switched system \eqref{DSwFDS}-\eqref{DSwFDSk} is positive or, equivalently,   $\mathcal M(A^0_k)= A^0_k$ and $\eta_k$ are increasing on $[-h,0],$ for each $k\in \underline N$.  Since $\eta_k(-h)=0$ it follows $ V(\eta_k)=\eta_k(0), \ \forall k\in \underline N$. Therefore we have
\begin{corollary}\label{positive}
	Let the time-delay  switched linear system  \eqref{DSwFDS}-\eqref{DSwFDSk} be positive.
	If there exists strictly positive vector $ \xi \gg0$ such that	
	\begin{equation} \label{cond2a}
		\big(A^0_k + \eta_k(0)\big)\xi \ll 0, \  \forall k \in \underline N,
	\end{equation}	
	then the system is exponentially stable under arbitrary switching.  	
\end{corollary}

\begin{remark}
	\label{re1}
	{\rm  The condition \eqref{cond2} means that the dual positive linear subsystems $\dot x(t)=\! \big(A^0_k + \eta_k(0)\big)^{\top}\!x(t), t\geq 0,$ $ k\in \underline N$ has the common LCLF of the form $L(x)=\xi^{\top}x$.  In the case of non-delay systems, i.e. when $\eta_k\equiv 0$, this result is given in \cite{Blanchini2015} (Proposition 3.4). Note additionally that in order to check whether or not Metzler matrices $ \big(\mathcal{M}(A^0_{k})  + V(\eta_k)\big)^{\top}, k\in \underline N$ (in Theorem \ref{main_lemma}) and   $\big(A^0_k + \eta_k(0)\big)^{\top}, \ k\in \underline N $  (in Corollary \ref{positive})  share  a common LCLF, one can use the procedure given by Theorem 4 in \cite{Knorn_Mason2009}.   
	}
\end{remark}
\begin{remark}\label{NAHS}
	{\rm  
It is worth mentioning here some relation between the above stability results and known results in the literature:

(i) Observe first that in the case of non-delay and non-switched systems (i.e. $\eta \equiv 0, N=1 $), Theorem \ref{main_lemma} is straightforward from Lemma \ref{lemma2.1} and \eqref{muA}.

(ii) In the case of delay non-switched systems ($N=1$) Theorem \ref{main_lemma} is an extension of the well-known result  on global asymptotic stability of positive linear systems with discrete delays (see, e.g. \cite{Chelaboina2004}) to linear FDEs which particularly cover both discrete delays and distributed delays and, moreover, are not necessarily positive. Indeed, the delay system
\begin{equation}\label{onedelay}
	\dot x(t) = A_1^0x(t) + A_1^1x(t-h) + \int_{-h}^0 B_1(\theta)x(t+\theta)d\theta, \ t\geq 0
\end{equation}
	 can obviously be represented in the form \eqref{DSwFDSk} (with $k=1$), by choosing 
	 \begin{equation*}
	\eta_1= \eta_1^{disc} + \eta_1^{dist}, \ \ \eta_1^{disc}(\theta)=\begin{cases} 0 \ &\text {if}\ \  \theta=-h,\\ A^1_1\ &\text{if}\ \  \theta \in (-h,0],       \end{cases} \ \ \text{and} \ \  \ \eta_1^{dist}= \int_{-h}^{\theta}B_1(s)\theta(t+s)ds.
	 \end{equation*}
 Therefore, by Theorem \ref{main_lemma}, the delay system \eqref{onedelay} is exponentially stable if \eqref{cond2} (with $k=1$) holds for some vector $\xi \gg 0$. If, moreover, the system \eqref{onedelay} is positive, or equivalently, $A^0_1$ is Metzler  and $A^1_1\geq 0, B_1(\theta)\geq 0, \forall \theta \in [-h,0]$ then the condition  \eqref{cond2} take the form $\big(A^0_1 +V(\eta_1)\big)\xi = \big(A_1^0+ A^1_1+ \int_{-h}^0B_1(s)ds\big)\xi \ll 0 $ which is reduced to the main result of \cite{Chelaboina2004} (Theorem 3.1) if $B_1\equiv 0$, i.e. if the system has no distributed delay. It is important to emphasize that our proof based on the comparison principle while in \cite{Chelaboina2004}  the Lyapunov function method was used.  
  
(iii) If the linear switched system \eqref{DSwFDS}-\eqref{DSwFDSk} is non-delay (i.e. $\eta_k \equiv 0, \ \forall k\in \underline N $) then Theorem \ref{main_lemma} gets back to a result given in \cite{Son_Ngoc_IET} (Lemma 2). 

	(iv) If $N=1$ and the system \eqref{DSwFDS} is positive then the condition \eqref{cond2a} of Corollary \ref{positive} is also necessary for exponential stability. In fact, since $A^0_k+\eta_k(0)$ is a Metzler matrix, \eqref{cond2a}  is equivalent to $\mu\big(A^0_k+\eta_k(0)\big) <0$ (due to Lemma \ref{lemma2.1}) which, in turn, implies the exponential stability of \eqref{DSwFDS}, by Theorem 4.1 in \cite{Ngoc_Naito}. 

(v) In our recent work \cite{SonNgocNAHS}, a more  general result similar to Theorem \ref{main_lemma} has been obtained, also by the comparison principle, for a  class of nonlinear time-varying switched systems under switchings with average dwell time, which covers the above result as a particular case. The proof of Theorem \ref{main_lemma} is much simpler and thus is given here for the sake of completeness of presentation.} 
\end{remark}

As the most important particular case of Theorem  \ref{main_lemma}, let us consider a class of switched linear system with multiple discrete time delays and distributed time delays of the form
\begin{align}
\label{discretedelaySwS}
\dot{x}(t) =  A^0_{\sigma(t)}x(t) + \sum_{i=1}^{m_{\sigma(t)}}A^i_{\sigma(t)}x(t-h^i_{\sigma(t)}) +\int_{-h_{\sigma(t)}}^0 B_{\sigma(t)}(\theta)x(t + \theta)d\theta,
\  t\geq 0,\ \sigma  \in \Sigma_+, 
\end{align}
where, for each $k\in \underline {N},\ 0 = h^0_k < h^1_k <...< h^{m_k}_k$, matrices $ A^{i}_k \in \R^{n\times n}$  and matrix functions $B_k(\cdot) \in C([-h_k,0],\R^{n\times n})$ are given. Define  $h:= \max_{k\in \underline {N}, i\in \underline{m_k}} \{ h^i_{k}, h_k\}, m:=\max\{m_k, k\in \underline N\}$ and set, for each $k\in \underline N, B_k(\theta)\equiv 0, \theta\in [-h,-h_k)$ (if $h_k<h$)  and $A^i_k=0$ for $i=m_{k+1}, \ldots, m $ (if $m_k<m$), then as in Remark \ref{NAHS} (ii),  the system \eqref{discretedelaySwS} can be represented in the form \eqref{DSwFDS}, with 
\begin{equation}
\label{eta_k}
\eta_k(\theta)=\sum\limits_{i=1}^{m} A_k^i \chi_{(-h_k^i, 0]}(\theta)+\int_{-h}^{\theta}B_k(s)ds,  k\in \underline{N},
\end{equation}
where $ \chi_M  $ denotes the characteristic function of a set $M\subset \R: \chi_M(\theta)=1$ if $\theta \in M $ and $\chi_M(\theta)=0$ otherwise. Since, obviously, 
\begin{equation*}
\label{var_eta_k}
V(\eta_k)\leq \sum\limits_{i=1}^{m} |A_k^i|+\int_{-h}^{0} |B_k(s)|ds,
\end{equation*}
we get therefore, by Theorem \ref{main_lemma},
\begin{corollary}\label{main_discretedelay} 
	Consider the time-delay switched linear system \eqref{discretedelaySwS}. If  there exists $\xi\in\R^n, \xi\gg0$ satisfying	
	\begin{equation}
	\label{cond3}
	\big(\mathcal{M}(A^0_k)+ \sum\limits_{i=1}^{m} |A_k^i|+\int_{-h}^0|B_k(s)|ds\big) \xi \ll 0, \  \forall k \in \underline N,
	\end{equation}
	where $h,m$ are defined as above, then this system is exponentially stable under arbitrary switching.
\end{corollary}

Similarly, from Corollary \ref{positive}, we have  

\begin{corollary}\label{positive_discretedelay} 
	Consider the switched positive linear system  with delay \eqref{discretedelaySwS}. If  there exists $\xi\in\R^n, \xi\gg0$ satisfying	
	\begin{equation}
	\label{cond4}
	\big(A^0_k+ \sum\limits_{i=1}^{m} A_k^i+\int_{-h}^0B_k(s)ds\big) \xi \ll 0, \  \forall k \in \underline N,
	\end{equation}
	where $h,m$ are defined as above, then this system is exponentially stable under arbitrary switching.
\end{corollary}
\begin{remark}
	\label{re0}
	{\rm  It is obvious that the above proof works also in the case of continuous time-varying delays $h_k^i(t)$ and $h_k(t)$ satisfying $
		\sup_{t\geq 0, k\in \underline N, i\in \underline m}\{h_k^i(t),h_k(t) \} \leq h, 
		$
		for some finite $h >0$. Therefore,  Corollary \ref{positive_discretedelay} can be considered as a generalization of the main result proved  in \cite{Liu_Yu_Wang, Liu_Dang} (where $B_k(\cdot)\equiv 0, \forall k)$. 
	}
\end{remark}	

The following  consequence of Theorem \ref{main_lemma} gives a sufficient condition for exponential stability of {\it a set of time-delay switched linear systems} of the form \eqref{DSwFDS}. The proof is straightforward, by the Remark \ref{NAHS} (iii).

\begin{corollary}
	\label{Corollary1} 
	Let a Metzler matrix $A_0\in \R^{n\times n}$ and an increasing matrix function  $\eta_0\in NBV([-h,0],\R^{n\times n})$ be given and the positive linear system $(A_0,\eta_0)$ of the form 
	\begin{equation}
	\label{DFDS}
	\dot{x}(t)=A_0x(t)+\int_{-h}^0d[\eta_0(\theta)]x(t+\theta), \  t\geq 0,
	\end{equation}
	be exponentially stable. Then, all time-delay switched linear systems of the form \eqref{DSwFDS} satisfying  
	\begin{equation}
	\label{condcoro1}
	\mathcal{M}(A^0_k)\leq A_0, \ V(\eta_k)\leq V(\eta_0),\  \forall k\in \underline{N},
	\end{equation}
	are exponentially stable under arbitrary switching.
\end{corollary}
We conclude this section by the following remark.
\begin{remark}\label{conservatism}{\rm 
Theorem \ref{main_lemma} and its corollaries give rather restrictive  sufficient conditions for exponential stability.  Basically, it is asserted that a linear switched delay system  is exponentially stable under arbitrary switching $\sigma \in \Sigma_+$ if its associate upper bounding positive subsystems possess a common linear copositive Lyapunov function (LCLF ). The  conservatism of these results, therefore, come from two sources: the use of upper bounding positive systems and the assumption on the existence of a common LCLF (which is given by the condition \eqref{cond2} or \eqref{cond2a} or similars). For instance, it is easy to give an example of two Metzler  matrices $A_1, A_2$ which are  Hurwitz stable but have no common vector $\xi\gg 0$ satisfying $A_k\xi\ll 0, k=1,2$ (see, e.g. \cite{Blanchini2015}, p.153). While the first source of conservatism is automatically removed when the original system is itself positive, one way to reduce the conservatism of the obtained results is to assume that vector $\xi$ in these conditions maybe different, depending on $k$. For instance, one can replace \eqref{cond2} by: there exist positive vectors $\xi_k\gg 0$ such that
\begin{equation}\label{cond2k}
	\big(\mathcal{M}(A^0_k)+ V(\eta_k)\big)\xi_k \ll 0, \  \forall k \in \underline N.
	\end{equation}
It is remarkable that if the system \eqref{DSwFDS}-\eqref{DSwFDSk} is positive then the relaxed condition of \eqref{cond2a}: $ \exists \xi_k\gg 0\   \text {s.t.} \  \big(A^0_k + \eta_k(0)\big)\xi \ll 0, \  \forall k \in \underline N, $ is also necessary for exponential stability under arbitrary switching, as mentioned in Remark \ref{NAHS} (iv). Therefore the aforementioned assumption is also quite natural. The cost to pay, however, is that the arbitrariness of switching then becomes smaller. Actually, in our recent work  \cite{SonNgocNAHS} it has been shown that  under such  less restrictive condition \eqref{cond2k},  the switched system \eqref{DSwFDS}-\eqref{DSwFDSk} is exponentially stable under arbitrary switchings {\it with average dwell time} (ADT) $ \sigma \in \Sigma_{\tau} \subset \Sigma_+$, for some $\tau >0 $. Similar results, in particular cases, can be found also in  \cite{Qi, Dong2015, Li2018}. In this paper we consider stability problem under arbitrary switching $\sigma\in \Sigma_+$ so that the sufficient conditions of the form \eqref{cond2} are proved that will be used in the next section for robustness analysis of stability.  Of course, similar problems can be considered for exponential stability under arbitrary switching with ADT, but we leave this for a future study.   		
	}
\end{remark}

\section{Bounds for the structured stability  radius }
Assume that the time-delay switched linear system \eqref{DSwFDS}-\eqref{DSwFDSk} is  exponentially  stable  under arbitrary switching $\sigma\in \Sigma_+$  and  the matrices $A^0_k,\eta_k(\cdot),\; k\in \underline{N}$ of the constituent subsystems  are subject to structured  affine perturbations of the form 
\begin{equation} \label{perturb-k}
A^0_k \rightarrow \widetilde A^0_k := A^0_k + D^0_k\Delta_kE^0_k,\  k\in \underline N\ \ \ \text{and}\ \ \ \eta_k(\cdot) \rightarrow \widetilde \eta_k(\cdot):= \eta_k(\cdot)+D^1_k\delta_k(\cdot)E^1_k, k\in \underline N. 
\end{equation}
Here, for each $k\in \underline N,\ D^0_k\in \R^{n\times r_k}, \ E^0_k\in \R^{q_k\times n},$ $ D^1_k\in \R^{n\times s_k},\ E^1_k\in \R^{p_k\times n}$ are  given matrices defining the structure of the perturbations,  $\Delta_k\in \R^{r_k\times q_k},\delta_k\in NBV([-h,0],\R^{s_k\times p_k}), $ $k \in \underline{N}$ are  unknown disturbances. For the sake of brevity, let us denote 
\begin{equation*}
\mathcal{D} :=\{(D^0_k, D^1_k),  k \in \underline{N}\},\ \mathcal{E}:=\{(E^0_k, E^1_k),  k\in \underline{N}\}.
\end{equation*} 
Then the perturbed switched system is described by 
\begin{equation}
\label{pertDSwFDS}
\dot x(t)=\widetilde A^0_{\sigma(t)}x(t)+\int_{-h}^0d[\widetilde \eta_{\sigma(t)}(\theta)]x(t+\theta), t\geq 0,\sigma \in\Sigma_+.
\end{equation}
As is well-known (see e.g. \cite{H-P2, H-P}  for the case of linear systems with no delays), by choosing appropriate structuring matrices $(D^0_k, D^1_k), (E^0_k, E^1_k) $  one can represent the perturbations which affect independently all elements of the $k$-th
subsystem matrices $A^0_k, \eta_k(\cdot)$  or their individual rows, columns or elements. The question we are interested in is, given the perturbation's structures $\mathcal D, \mathcal E$, how  large  disturbances  $\Delta_k,\ \delta_k(\cdot),\; k\in \underline N$ are allowable  without  destroying  the  exponential stability  of  the  perturbed system \eqref{pertDSwFDS}, subject to arbitrary switching signals $\sigma \in \Sigma_+$. To this end, let us measure the size of disturbances 
${\bf\Delta} :=\{[\Delta_k,\ \delta_k(\cdot)], \ k\in \underline N\}$ by the quantity
\begin{equation}
\label{normdelta}
\|{\bf \Delta}\|_{max}:=\max_{k\in \underline{N}}(\|\Delta_k\|+\|\delta_k\|).
\end{equation}
Then the robustness of exponential stability of the system \eqref{DSwFDS}-\eqref{DSwFDSk} can be quantified by the following  definition.
\begin{definition}
	\label{stabrad}
	Assume that the time-delay switched linear system \eqref{DSwFDS}-\eqref{DSwFDSk} is exponentially stable under arbitrary switching. Then its structured stability radius w.r.t. parameters affine perturbations of the form \eqref{perturb-k} is defined as \begin{align}
	\label{stabraddefinit}
	r_{\R}(\mathcal{A},\Gamma,\mathcal{D},\mathcal{E}) :=\inf\big\{\|{\bf\Delta}\|_{max}: \exists\sigma \in \Sigma_+ \text{ s. t. the perturbed system}\  \eqref{pertDSwFDS}\ \text{is not exponentially stable} \big\}.
	\end{align}
	If $D_k^i=E^i_k= I_n, \forall k\in \underline N,  i=0,1$  then we put $ r_{\R}(\mathcal{A},\Gamma)= r_{\R}(\mathcal{A},\Gamma,\mathcal{D},\mathcal{E})$ and call this quantity  unstructured stability radius of the system \eqref{DSwFDS}- \eqref{DSwFDSk}.  	
\end{definition}
It follows from Definition \eqref{stabraddefinit} that the perturbed switched system \eqref{pertDSwFDS} is  exponentially stable under arbitrary switching $\sigma \in \Sigma_+$ for any disturbance ${\bf\Delta} :=\{[\Delta_k,\ \delta_k(\cdot)], \ k\in \underline N\}$ satisfying $\|{\bf \Delta}\|_{max}=\max_{k\in \underline{N}}(\|\Delta_k\|+\|\delta_k\|) < 	r_{\R}(\mathcal{A},\Gamma,\mathcal{D},\mathcal{E}).$

\begin{remark}
	\label{remark3}
	{\rm  It follows from Definition \eqref{stabraddefinit} that the perturbed switched system \eqref{pertDSwFDS} is  exponentially stable under arbitrary switching $\sigma \in \Sigma_+$, for any disturbance ${\bf\Delta} :=\{[\Delta_k,\ \delta_k(\cdot)], \ k\in \underline N\}$ satisfying $\|{\bf \Delta}\|_{max}
			 < 	r_{\R}(\mathcal{A},\Gamma,\mathcal{D},\mathcal{E})$}.
	
\end{remark}
		
\begin{remark}\label{stabrad-k}{\rm
Obviously, if the switched system \eqref{DSwFDS} has only one subsystem $(A^0_k,\eta_k)$, for a fixed $k\in \underline N$,  then the above definition is reduced to the well-known notion of the real structured stability radius $r_{\R}(A^0_k, \eta_k, \mathcal{D}_k, \mathcal{E}_k)$ of the  time-delay subsystem  $(A^0_k,\eta_k)$ subject to structured affine perturbations \eqref{perturb-k}, that was studied in \cite{Son_Ngoc2001, Ngoc_Son_SIAM}, namely,
	\begin{equation}\label{stabradk}
	r_{\R}(A^0_k, \eta_k, \mathcal{D}_k, \mathcal{E}_k)= \inf\big\{\|\Delta_k\|+ \|\delta_k\|: (\widetilde A^0_{k}, \widetilde \eta_{k})\ \ \text {is not exponentially stable} \big\}.	
		\end{equation}
In particular, we can formulate the following result in \cite{Son_Ngoc2001} for later use: If the time-delay system $(A,\eta)$ is positive (i.e. $A$ is Metzler and $\eta$ is increasing on $[-h,0]$) and subjected to affine perturbations $A\rightarrow \widetilde A:= A+D^0\Delta E^0, \eta \rightarrow \widetilde \eta := \eta+D^1\delta E^1$ with {\it nonnegative} structuring matrices $D^i,E^i, i=0,1$ then its structured stability radius $r_{\R}:= r_{\R}(A,\eta, \{D^0,D^1\},\{E^0,E^1\})$ satisfies the following estimates
		$$
		\frac{1}{\max_{i,j \in \{0,1\}}\|G_{i,j}(0)\|}\leq r_{\R} \leq 	\frac{1}{\max_{i \in \{0,1\}}\|G_{i,i}(0)\|}.
		$$ 
		where $G_{i,j}(s)= E^i\big(P(s)\big)^{-1}D^j $ are the transfer matrix functions and $ P(s)= sI-A-\int_{-h}^0e^{s\theta}d[\eta(\theta) ] $ is the characteristic quasi-polynomial matrix of the system. 
		In particular, if $D^0=D^1$ or $E^0=E^1$, then we get
		\begin{align}
		\label{realstabrad}
		r_{\R}(A,\eta, \{D^0,D^1\},\{E^0,E^1\}) =\frac{1}{\max_{i\in\{0,1\}}\|G_{i,i}(0)\|}
		= \frac{1}{\max_{i\in \{0,1\}}\|E^i(P(0))^{-1}D^i\|}.
		\end{align}	
		It follows, in particular,  that the {\it unstructured} stability radius (i.e. when $D^i=E^i=I_n, i=0,1$ ) of a positive time-delay system $(A,\eta)$ can be calculated  by the simple formula
		\begin{equation}\label{unststabrad}
		r_{\R}(A,\eta) = \|(P(0))^{-1}\|^{-1}=\|(A+\eta(0))^{-1}\|^{-1} .
		\end{equation}	
}
\end{remark}
Turning back to the time-delay switched linear system \eqref{DSwFDS}-\eqref{DSwFDSk}, assume  that the condition \eqref{cond2} holds or, equivalently, the open convex cone
\begin{equation*}
\mathcal{G}_{\mathcal{A},\Gamma}=\big\{ \xi \in \R^n_+: \xi\gg 0, \big(\mathcal{M}(A^0_k)+V(\eta_k)\big)\xi\ll 0,k\in \underline N \big\} 
\end{equation*}
is non-empty. Then, by Theorem \ref{main_lemma}, the switched linear system \eqref{DSwFDS}-\eqref{DSwFDSk} is exponentially  stable. Denote, for each  $ \xi \in \mathcal{G}_{\mathcal{A}, \Gamma}$ and each $k\in \underline N, $ 
\begin{equation}\label{beta0}
\big(\mathcal{M}(A^0_k)+V(\eta_k)\big)\xi = - (\beta^k_1(\xi), \beta^k_2(\xi), \ldots, \beta^k_n(\xi))^{\top},  
\end{equation} 
and define
\begin{equation}
\label{beta}
\beta(\xi) := \min_{k\in \underline{N}, i\in \underline n} \beta^k_i(\xi),
\end{equation}
then, clearly, $\beta(\xi) >0$. The following theorem is the main result of the paper.

\begin{theorem}
	\label{theorem4.1} 
Let the switched linear system \eqref{DSwFDS}-\eqref{DSwFDSk} be exponentially stable under arbitrary switching $\sigma\in \Sigma_+$. Assume, moreover, that the condition \eqref{cond2} holds or, equivalently, $\mathcal{G}_{\mathcal{A}, \Gamma} \not=\emptyset $. Then the stability radius of the switched linear system \eqref{DSwFDS}-\eqref{DSwFDSk} subjected to structured affine perturbations \eqref{perturb-k} satisfies the inequality
	\begin{equation}
	\label{estim_1norm}
	\frac{1}{M_0} \sup\limits_{\xi\in \;\mathcal{G}_{\mathcal{A}, \Gamma}}\dfrac{\beta(\xi) }{ \|\xi\|} \leq  	r_{\R}(\mathcal{A},\Gamma,\mathcal{D},\mathcal{E})\leq  \min_{k\in \underline N} r_{\R}(A^0_k, \eta_k, \mathcal{D}_k, \mathcal{E}_k), 
	\end{equation}	
	where	
	\begin{equation*}	\label{M}
	M_0 :=\max\limits_{k\in \underline{N}}\{\|D^0_k\| \|E^0_k\|; \|D^1_k\| \|E^1_k\| \} .
	\end{equation*}
\end{theorem}
\begin{proof}
To prove the upper bound, assume to the contrary that $r_{\R}(\mathcal{A},\Gamma,\mathcal{D},\mathcal{E}) > \min\limits_{k\in \underline N} r_{\R}(A^0_k,\eta_k, \mathcal{D}_k,\mathcal{E}_k) =  r_{\R}(A^0_{k_0},\eta_{k_0}, \mathcal{D}_{k_0},\mathcal{E}_{k_0}),$ for some $k_0\in \underline N$.
By the definition \eqref{stabradk} of the  structured stability radius  of the subsystem $k_0$, there exists a perturbation $[\Delta_{k_0}, \delta_{k_0}]$ such that  
$$ r_{\R}(\mathcal{A},\Gamma,\mathcal{D},\mathcal{E}) > \|\Delta_{k_0}\|+\|\delta_{k_0}\|> r_{\R}(A^0_{k_0},\eta_{k_0}, \mathcal{D}_{k_0},\mathcal{E}_{k_0}),  
$$
and the time-delay perturbed  system $(\widetilde{A}^0_{k_0}, \widetilde{\eta}_{k_0})$
(with $\widetilde{A}^0_{k_0}:=A^0_{k_0}+D^0_{k_0}\Delta_{k_0}E^0_{k_0},\  \  \widetilde{\eta}_{k_0}(\cdot) := \eta_{k_0}(\cdot) + D^1_{k_0}\delta_{k_0}(\cdot)E^1_{k_0}$) is not exponentially  stable. This implies, however, that the perturbed switched linear system \eqref{pertDSwFDS} is not exponentially stable under the switching signal $\sigma (t)\equiv k_0,\ t\geq 0$, contradicting the definition of stability radius \eqref{stabraddefinit}, in view of Remark \ref{remark3}.

We can deduce, for each $k\in \underline N, \xi \in \mathcal{G}_{\mathcal{A}, \Gamma} $ and arbitrary disturbances $ \Delta_k\in \R^{r_k\times q_k},\ \delta_k(\cdot)\in NBV([-h,0],\R^{s_k\times p_k}),$ 
\begin{align}
\big (\mathcal{M}(\widetilde{A}^0_k)+V(\widetilde{\eta}_k)\big )\xi 
&\leq \big (\mathcal{M}(A^0_k)\xi+V(\eta_k)\xi+ \big (|D^0_k\Delta_kE^0_k|+V(D^1_k\delta_kE^1_k)\big)\xi\notag  \\   
&\leq -\beta(\xi){\bf 1}_n + \big(|D^0_k\Delta_kE^0_k| + V(D^1_k\delta_kE^1_k)\big )\xi, \label{estim5}
\end{align}
\noindent where ${\bf 1}_n := (1,1,\ldots,1)^\top$ and $\beta(\xi)$ is defined by \eqref{beta}. Using \eqref{ineqvar}, \eqref{ineqvar1} and the definition of $M_0$  we get easily, for any $k\in \underline N$, 
\begin{equation}
\label{estim10}\||D^0_k\Delta_kE^0_k| + V(D^1_k\delta_kE^1_k) )\xi \| \leq M_0( \|\Delta_k\|+\|\delta_k\| )\|\xi\|.  
\end{equation}
It follows that, for any disturbance $ {\bf\Delta } :=\{[\Delta_k, \delta_k(\cdot)],k\in \underline N\}$ satisfying
\begin{equation}
\label{pertubestim}
\|{\bf\Delta}\|_{max}:=\max_{k\in \underline{N}}(\|\Delta_k\|+\|\delta_k\|)< \dfrac{1}{M_0}\dfrac{\beta(\xi) }{ \|\xi\|}, 
\end{equation}
we have  $  \|(|D^0_k\Delta_kE^0_k| + V(D^1_k\delta_kE^1_k))\xi \|< \beta(\xi). $  Therefore
every component of the vector $\big(|D^0_k\Delta_kE^0_k| + V(D^1_k\delta_kE^1_k)\big )\xi   $ 
is strictly smaller than $\beta(\xi)$.  It follows,  by \eqref{estim5},  that 
$$
\big (\mathcal{M}(\widetilde{A}^0_k)+ V(\widetilde{\eta}_k)\big )\xi \ll 0, \ \forall k \in \underline N. 
$$
By Theorem \ref{main_lemma}, the perturbed system \eqref{pertDSwFDS} is exponentially stable under arbitrary switching $\sigma\in \Sigma_+$. Since this is proved for any disturbance ${\bf\Delta}$ satisfying \eqref{pertubestim}, we have,  by definition,  $r_{\R}(\mathcal{A},\Gamma,\mathcal{D},\mathcal{E})\geq \dfrac{1}{M_0} \dfrac{\beta(\xi) }{ \|\xi\|}$, for any $\xi \in \mathcal{G}_{\mathcal{A},\Gamma}, $   yielding the lower bound in \eqref{estim_1norm} and completing the proof.
\end{proof}

\begin{remark}
	{\rm 
We note that the calculation of the lower bound of stability radius in \eqref{estim_1norm} requires solving a system on $N$ linear inequalities to form the open convex cone $\mathcal{G}_{\eta,\Gamma}$  of all strictly positive solutions and then solving the optimization problem $ \sup\big\{\dfrac{\beta(\xi)}{\|\xi\|}:\   \xi \in \mathcal{G}_{\eta,\Gamma} \big\}.$ Moreover, since  $\beta(\xi)/\|\xi\| = \beta(\xi/\|\xi\|)$ the last problem is obviously reduced to finding the maximum of the function $\beta (\cdot)$ over the compact set $\text{cl}\mathcal{G}_{\eta,\Gamma}\cap S_1$ where $S_1=\{\xi \in \R^n: \|\xi\|=1\}\ $-the unit ball of $\R^n$. Further, the lower bound in \eqref{estim_1norm} is calculated under a rather restrictive assumption \eqref{cond2} which is equivalent to the existence of common LCLF $L(x)= \xi^{\top}x$ for all linear positive subsystems $\dot x= \big(\mathcal{M}(A^0_k)+ V(\eta_k)\big)^{\top}x, k\in \underline N $, resulting  inevitably in conservatism of this estimating bound. Some possibility to improve this result is mentioned before, in Remark 4. In this regard, it is an interesting question to find out a particular class of time-delay switched linear systems for which the estimates \eqref{estim_1norm} yields actually a formula for calculation the stability radius (see \cite{Son_Ngoc_IET} Corollary 3, for the case of non-delay switched systems). 
	
}
	
\end{remark}

In the following example a numerical simulation in $\R^2$ is given to illustrate Theorem \ref{theorem4.1}.\\
\vskip 0.1cm
{\bf Example 1.}	
	{\rm Consider the time-delay  switched positive linear system \eqref{DSwFDS}-\eqref{DSwFDSk} in $\R^2 $ with $h=1, N=2, $ 
	\begin{equation*}
	A^0_1=
	\begin{bmatrix}
	-5.0221&0.2531\\
	1.0103&-3.0105
	\end{bmatrix},\ 	
	A^0_2 = 
	\begin{bmatrix}
	-4.1023& 0.2517\\
	0.5314&-2.4531
	\end{bmatrix},\
	\end{equation*}
	and, for $k=1,2,$ 
	\begin{equation*}
	\eta_k(\theta)=
	\begin{cases}
	0 \; \; \;  \; \; \; \;\; \; \; \; \quad \quad \text{if} \; \; \; \theta =-1;\\
	B_k\in \R^{2\times 2}\;\;\; \;  \text{if} \; \theta\in (-1,\;0] ,\; \; 
	\end{cases}
	\end{equation*}	
	\begin{align*}
	B_1=
	\begin{bmatrix}
	0.6321&0.3507\\
	1.0315&0.2403
	\end{bmatrix},\; B_2=
	\begin{bmatrix}
	1.103&0.5041\\
	0.7013&0.1102
	\end{bmatrix}.
	\end{align*}	
	Choosing $\xi_0=\begin{bmatrix}
	2 & 5\end{bmatrix}^{\top}\gg 0,$ we verify readily that $\xi_0$ satisfies \eqref{cond2}. Therefore, this time-delay switched linear system is exponentially stable. Assume that system's matrices  are subjected to structured perturbations so that the perturbed subsystems take the form 
	\begin{equation}
	\label{exampertDSw}
	\dot x(t)=\widetilde{A}^0_k(t)+\widetilde{B}_kx(t-1), t\geq 0, k=1,2,
	\end{equation}
	where
	\begin{align*}
	\widetilde{A}^0_1=
\begin{bmatrix}
		-5.0221& 0.2531\\ 
		1.0103+\delta_1&-3.0105+\delta_2
		\end{bmatrix},
	\widetilde{B}_1 = 
\begin{bmatrix}
		0.6321+\gamma_1&0.3507+\gamma_2\\
		1.0315&0.2403
		\end{bmatrix},
	\end{align*}
	\begin{align*}
\widetilde{A}^0_2= 
		\begin{bmatrix}
		-4.1023+\delta_3& 0.2517+\delta_4\\
		0.5314&-2.4531
		\end{bmatrix},
		\widetilde{B}_2=
		\begin{bmatrix}
		1.103&0.5041\\
		0.7013+\gamma_3&0.1102+\gamma_4
		\end{bmatrix},
	\end{align*}
	and $ \delta_k, \gamma_k, k\in \underline 4$ are unknown disturbances.  Then, taking the structuring matrices $ 		
	D^0_1 = D^1_2= [0\ 1]^{\top}, \ D^0_2=D^1_1=[1\ 0]^{\top}  $ and $	E^0_1 = E^0_2 = E^1_1= E^1_2 = I_2$ (the identity matrix in $\R^{2\times 2}$) we can represent this perturbation model in the form \eqref{perturb-k}. 
	Since all subsystems are positive and, for characteristic quasi-polynomial matrices $P_k(s), k=1,2$,
	\begin{align*} 
	P_1(0)=A^0_1+\eta_1(0)= A^0_1 + B_1 &=
	\begin{bmatrix}
	-4.39&0.6038\\
	2.0418&-2.7702
	\end{bmatrix},\\
	P_2(0)=A^0_2+\eta_2(0)=A^0_2 + B_2 &=
	\begin{bmatrix}
	-2.9993&0.7558\\
	1.2327&-2.3429
	\end{bmatrix},
	\end{align*}	
	we can use \eqref{beta} to compute their real stability radii and get the upper bound in \eqref{estim_1norm} as
	 $$\min_{k\in \underline N} r_{\R}(A^0_k, \eta_k, \mathcal{D}_k, \mathcal{E}_k) = 2.0323. $$
By  \eqref{realstabrad} and \eqref{beta0} we get readily $\beta(\xi_0)=\min\limits_{k\in \underline{N}, k\in \underline n} \beta^k_i(\xi_0)=2.2196.$ Moreover,  clearly,  $M_0= \max\limits_{k=1,2}\{\|D^0_k\|\|E^0_k\|,\; \|D^1_k\|\|E^1_k\|\} =1.$

	\noindent	
	Therefore, by Theorem \ref{theorem4.1} we obtain the following lower bound for the stability radius of the time-delay switched linear system under consideration: 
	\begin{equation*}	
	r_{\R}(\mathcal{A},\Gamma,\mathcal{D},\mathcal{E})\geq 	\frac{1}{M_0} \sup\limits_{\xi\in \;\mathcal{G}_{\mathcal{A}, \Gamma}}\dfrac{\beta(\xi) }{ \|\xi\|} \geq \dfrac{\beta(\xi_0) }{ \|\xi_0\|}=0.4439.
	\end{equation*}
		
\noindent	It follows, by definition, that the  perturbed time-delay switched linear system  associated with \eqref{exampertDSw} is  exponentially stable, under arbitrary switching $\sigma\in \Sigma_+$,  for any disturbance ${\bf\Delta}$  satisfying $\|{\bf\Delta}\|=\max\limits_{k\in \underline 4}\{|\delta_k|+|\gamma_k|\}< 0.4439$. Thus, if we choose randomly disturbance parameters $\delta_k,\gamma_k, k\in \underline 4$ satisfying this condition and an arbitrary switching $\sigma \in \Sigma_+ $  then the trajectory of the perturbed switched system (simulated by MATLAB toolbox) decays exponentially to zero as $t$ tends to the infinity, as shown in Figure 1. Contrarily, if we choose disturbance parameters $\delta_1=5.012,\delta_2=1.001, \delta_3=0.2005, \delta_4=1.0102$ and $\gamma_1=2.002,\gamma_2=1.901, \gamma_3=2.012, \gamma_4=3.1023$ (so that $\|{\bf\Delta }\|_{max}\geq 2.0323$, a upper bound of the stability radius $r_{\R}({\mathcal A,\Gamma, \mathcal D, \mathcal E})$ ) and a periodic switching signal 
	\begin{equation*}
	\sigma(t)=
	\begin{cases}
	1 & \text{if} \;  lT\leq t< 2+lT,\\
	2 & \text{if}\; 2+lT\leq t<3+lT, \; l =0,1,\dots, 
	\end{cases}
	\end{equation*}
	then the corresponding trajectory  of the perturbed switched system does not go to zero when $t \rightarrow \infty$,  as shown in Fig 2.
		\begin{figure}[ht]
		\begin{center}
			\includegraphics[height=6cm,width=17cm]{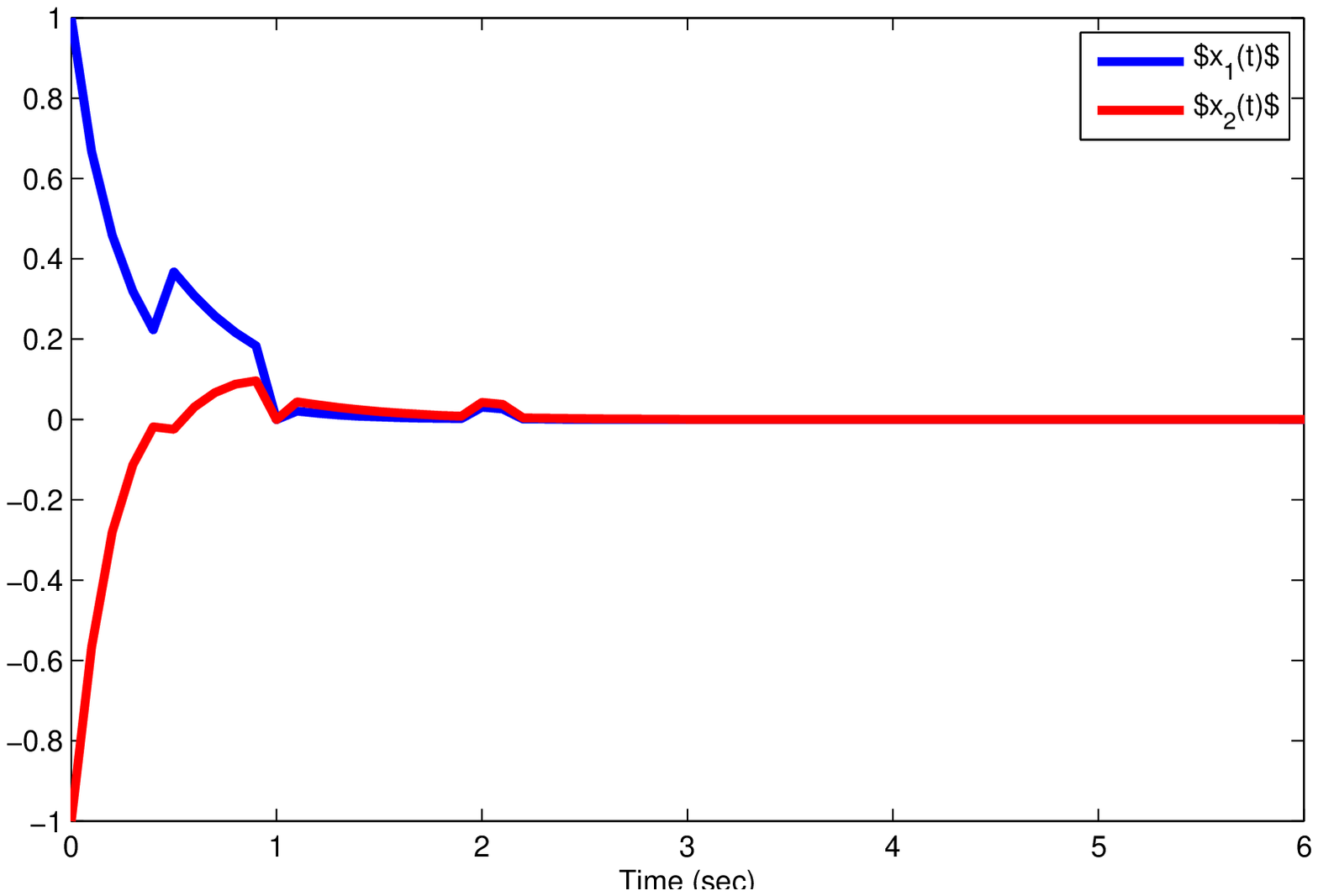}
		\end{center}
		\vskip -0.4cm		
		\caption{\textit{All perturbed systems are exponentially stable under arbitrary switching} if $\|{\bf\Delta }\|_{max}<\dfrac{1}{M_0} \sup\limits_{\xi\in \;\mathcal{G}_{\mathcal{A}, \Gamma}}\dfrac{\beta(\xi) }{ \|\xi\|}$}
	\end{figure}
	\vskip -1.0cm
	\begin{figure}[ht]
		\begin{center}
			\includegraphics[height=6cm,width=17cm]{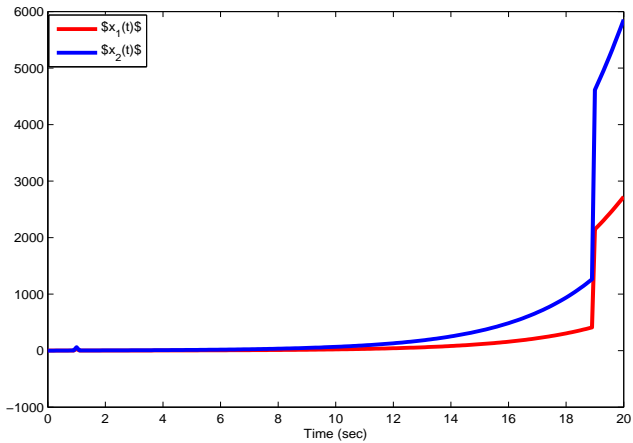}
		\end{center}
		\vskip -0.5cm
		\caption{\textit{Perturbed system may become unstable, for some switching law, if $\|{\bf\Delta }\|_{max}\geq \min_{k\in \underline N} r_{\R}(A^0_k, \eta_k, \mathcal{D}_k, \mathcal{E}_k)$ }}
	\end{figure}
	
 Below we will make use of  Corollary \ref{Corollary1} in Section 3  and  Corollary 6 in \cite{Son-Hinrichsen96} to get a more explicit  formula  for  computing this lower bound, under some additional assumptions.

\begin{theorem}
	\label{theorem4.3}
	Assume that the switched linear system \eqref{DSwFDS}-\eqref{DSwFDSk} is subject to affine perturbations of the form  \eqref{perturb-k}  with $r_k=r_0,$ $ q_k=q_0, s_k=s_0, p_k=p_0$ for all $k\in \underline N$. Assume, moreover, that  there exist a  Metzler matrix $A_0\in \R^{n\times n}$, an increasing matrix  function $\eta_0(\cdot)\in NBV([-h, 0],\mathbb{R}^{n\times n})$   and nonnegative matrices
	$D_0\in\R_+^{n\times r_0},E_0\in\R_+^{q_0\times n},$ $D_1\in\R_+^{n\times s_0}, E_1\in\R_+^{p_0\times n},$
	such that $A_0+\eta_0(0)$ is Hurwitz stable and  	
	\begin{align}
	&\mathcal{M}(A^0_k)\leq A_0, \ V(\eta_k)\leq V(\eta_0), \label{Con0} \\
	& |D^0_k|\leq D^0, |E^0_k|\leq E^0,|D^1_k|\leq D^1, |E^1_k|\leq E^1,\label{Con2}
	\end{align}
	\noindent 
	for all $ k\in \underline N$. Then the real stability radius of the time-delay switched linear system \eqref{DSwFDS}-\eqref{DSwFDSk} satisfies the following estimate
	\begin{equation}\label{case1normA0eta}
	r_{\R}({\mathcal A,\Gamma, \mathcal D, \mathcal E})\geq \dfrac{1}{M_0\|(A_0+\eta_0(0))^{-1}{\bf 1}_n \|},
	\end{equation}
	where $M_0 := \max\{\|D^0\|\|E^0\|; \|D^1\|\|E^1\|\}$ and ${\bf 1}_n := (1,1,\cdots, 1)^{\top}$.
\end{theorem}
\begin{proof}
In view of \eqref{Con0}, it follows from Corollary \ref{Corollary1} that the system \eqref{DSwFDS}-\eqref{DSwFDSk} is exponentially stable  under arbitrary switching. Further, by Lemma \ref{lemma2.1} (iii), the Metzler matrix  $A_0+\eta_0(0)$ is invertible and $-(A_0+\eta_0(0))^{-1} \geq 0$. This implies that $ -(A_0+\eta_0(0))^{-1}$ maps the interior of $\R^n_+$ into itself and, consequently, $ \xi_0:=-(A_0+\eta_0(0))^{-1}{\bf 1}_n  \gg 0  $. Thus, $(A_0+\eta_0(0))\xi_0 =-{\bf 1}_n \ll 0$. By \eqref{Con0} we have, for each $k\in \underline N$  and arbitrary disturbances $ \Delta_k\in \R^{r_0\times q_0},\ \delta_k(\cdot)\in NBV([-h,0],\R^{s_0\times p_0}),$ 
\begin{align}
\label{estim2}
\mathcal{M}(\widetilde{A}^0_k)+ V(\widetilde{\eta}_k) \leq \mathcal{M}(A^0_k)+V(\eta_k)+  |D^0_k\Delta_kE^0_k| +V(D^1_k\delta_kE^1_k)   
\leq (A_0+\eta_0(0))+ |D^0_k\Delta_0E^0_k| + V(D^1_k\delta_kE^1_k),
\end{align}
which implies 
\begin{align*} 
\big (\mathcal{M}(\widetilde{A}^0_k)+ V(\widetilde{\eta}_k)\big )\xi_0\leq -{\bf 1}_n + \big(|D^0_k\Delta_0E^0_k| + V(D^1_k\delta_kE^1_k)\big)\xi_0.
\end{align*}
Therefore, similarly as \eqref{estim10} and using \eqref{Con2}, we get    
\begin{align*}
\| \big(|D^0_k \Delta_kE^k_0| + V(D^1_k\delta_kE^1_k)\big)\xi_0\| \leq M_0\big(\|\Delta_k\| + \|\delta_k\|\big)\|(A_0+\eta_0(0))^{-1}{\bf 1}_n \|,\ k\in \underline N. 
\end{align*}
It follows  that for any disturbances $ {\bf\Delta }= \{[\Delta^0_k,\delta_k], k\in \underline N\}$ satisfying 
$$
\|{\bf\Delta}\|_{max}:=\max_{k\in \underline{N}}(\|\Delta_k\|+\|\delta_k\|)< \dfrac{1}{M_0\|(A_0+\eta_0(0))^{-1}{\bf 1}_n\|}
$$
we have, for each $k\in \underline N,\  
\big (\mathcal{M}(\widetilde{A}^0_k)+ V(\widetilde{\eta}_k)\big )\xi_0 \ll 0,$
with $\xi_0 \gg 0. $ This implies, by Theorem \ref{main_lemma}, that the perturbed switched system  is exponentially stable under arbitrary switching. Thus,  by the definition, the system stability radius $r_{\R}(\mathcal{A},\Gamma,\mathcal D,\mathcal E)$ satisfies the lower bound  \eqref{case1normA0eta}. The proof is complete.
\end{proof}

It is worth mentioning that Theorems \ref{theorem4.1},  \ref{theorem4.3} can be applied to give the lower bounds for stability radii of switched linear systems with discrete multi-delays and/or distributed delay of the form \eqref{discretedelaySwS}. In particular, for the class of positive switched systems with delay, the following consequence of Corollary \ref{Corollary1} and Theorem \ref{theorem4.3} gives explicit bounds of the unstructured stability radius for {\it a set of switched positive linear systems. }
\begin{corollary}
	\label{theorem4.5}
	Assume that the positive  linear system
	\begin{equation}
	\label{DFDS1}
	\dot{x}(t)=A^0_0x(t)+ \sum_{i=1}^{m}A^i_0x(t-h^i_0)+\int_{-h}^0B_0(\theta)x(t+\theta)d\theta,  t\geq 0,
	\end{equation}	
	(where $h^i_0\geq 0, h:= \max \{h^i_0, i\in \underline m\})$ is exponentially stable. Then, for any triples $(A^0_k , A^i_k, B_k(\cdot)) \in \R^{n\times n} \times \R^{n\times n}_+ \times  C([-h,0],\R^{n\times n}_+),$ for all $ i\in \underline m,  k\in \underline N, \theta \in [-h,0]$  satisfying
	\begin{align}\label{triple}
	\mathcal{M}(A^0_k)= A^0_k\leq A^0_0,\  A^i_k\leq A^i_0,\; B_k(\theta) \leq B_0(\theta),
	\end{align}
	the  switched positive linear system with delay 
	\begin{align}
	\label{SwDDS}
	\dot{x}(t)=A^0_{\sigma(t)}x(t)+\sum_{i=1}^{m}A^i_{\sigma(t)}x(t-h^i_0)+\int_{-h}^0B_{\sigma(t)}(\theta)x(t+\theta)d\theta, t\geq 0, \sigma\in \Sigma_+,
	\end{align}
	is exponentially stable under arbitrary switching. Moreover, for each of the switched systems \eqref{SwDDS} satisfying \eqref{triple}, the real unstructured stability radius under perturbations 
	\begin{align}
	\label{unstructpert}
	&A^0_k\rightarrow \widetilde A^0_k = A^0_k + \Delta^0 _k, A^i_k\rightarrow \widetilde A^i_k=A^i_k + \Delta ^i_k, \ i\in \underline m, \  k\in \underline N, \notag\\
	&	B_k(\theta)\rightarrow \widetilde B_k(\theta)= B_k(\theta) + C_k(\theta),\theta\in [-h,0], k\in \underline N,   
	\end{align}
	(with unknown $\Delta^0 _k, \Delta^i_k \in \R^{n\times n }, C_k(\cdot) \in C([-h,0], \R^n)$ ) 	satisfies the following estimates
	\begin{equation}
	\label{unstrurad}
	\dfrac{1}{\|H_0^{-1}\|}\ \leq \ 	r_{\R}({\mathcal A},\Gamma) \ \leq \   \dfrac{1}{\max_{k\in \underline N}\|H_k^{-1}\|}, 
	\end{equation}
	where $
	H_k :=  \sum_{i=0}^m  A^i_k + \int_{-h}^0B_k(\theta) d\theta , k=0,1,\ldots, N. $
\end{corollary}

\begin{proof}
First, as  mentioned in Remark \ref{NAHS} and Corollaries \ref{main_discretedelay}, \ref{positive_discretedelay}, the positive systems \eqref{DFDS1} and \eqref{SwDDS} can be represented, respectively, in the form \eqref{DFDS}
and \eqref{DSwFDS}, with some increasing matrix functions $\eta_k, \eta_0 \in NBV([-h,0],\R^{n\times n}), k\in \underline N$. Moreover,  by \eqref{triple}, $V(\eta_k) \leq V(\eta_0)=\eta_0(0), \forall k\in \underline N$. Therefore, by Corollary \ref{Corollary1}, the positive switched system \eqref{SwDDS} is exponentially stable under arbitrary switching. Further, the unstructured perturbation model \eqref{unstructpert} can be represented in the form $ A^0_k\rightarrow \widetilde A^0_k = A^0_k + \Delta^0 _k,\ \ \eta_k(\cdot)  \rightarrow \widetilde \eta_k(\cdot) = \eta_k(\cdot) + \delta_k(\cdot),k\in \underline N$, where $\delta_k\in NBV([-h,0],\R^{n\times n}) $ is unknown perturbation defined as
\begin{equation}
\label{unstructpert2}
\delta_k(\theta)=\sum\limits_{i=1}^{m} \Delta_k^i \chi_{(-h_0^i, 0]}(\theta)+\int_{-h}^{\theta}C_k(s)ds,\  k\in \underline N.
\end{equation}
Since $A_0^0+\eta_0(0) = H_0$ and $ \|H_0^{-1 }{\bf 1}_n\|\leq \|H_0^{-1}\|= \sup\{\|H_0^{-1}\xi\|, \|\xi\|=1\}$ (noticing that $\|{\bf 1}_n\|_{\infty}=1$),  the lower bound in \eqref{unstrurad} is followed  from Theorem \ref{theorem4.3}. On the other side, it follows from \eqref{unststabrad} that, for each $k\in \underline N$, the real unstructured stability radius of the positive linear system 
\begin{equation}
\label{k}
\dot{x}(t)=A^0_kx(t)+ \int_{-h}^0d[\eta_k(\theta]x(t+\theta),\  t\geq 0, 
\end{equation}
is given by the formula $ r_{\R}(A^0_k,\eta_k)= \|(P_k(0))^{-1}\|^{-1} = \|H^{-1}_k\|^{-1},
$
where $P_k(s)$ is the characteristic quasi-polynomial of  \eqref{k}: $P_k(s)=sI-A^0_k-\int_{-h}^0e^{s\theta}d[\eta_k(\theta)] = sI - A^0_k-\sum_{i=1}^me^{-s h_0^i}A_k^i - \int_{-h}^0 e^{s\theta}B_k(\theta)d\theta$. The upper bounds in \eqref{unstrurad} is now immediate from Theorem \ref{theorem4.1}.   This completes  the proof. 
\end{proof}
Below we give a simple example to illustrate Corollary \ref{theorem4.5}.\\
{\bf Example 2.}
Consider the time-delay positive  switched linear  system of the form \eqref{SwDDS} in $\R^3$  with $N=2,h=2,$ $m=3,$ where the system's matrices 
\begin{align*}
A^0_1=
\begin{bmatrix}
-18&1&0\\
1&-15&1\\
1&1&-13
\end{bmatrix},\ 
A^0_2 =
\begin{bmatrix}
-19&1&0\\
1&-14&1\\
1&1&-15
\end{bmatrix},
\end{align*}
\begin{align*}
A^1_1=
\begin{bmatrix}
1&1&0\\
1&0&1\\
1&0&1
\end{bmatrix},
A^2_1 =
\begin{bmatrix}
0&1&0\\
1&0&1\\
0&1&2
\end{bmatrix},
A^3_1 =
\begin{bmatrix}
1&0&1\\
1&0&1\\
0&1&0
\end{bmatrix}, 	A^1_2=
\begin{bmatrix}
0&1&1\\
1&0&1\\
1&0&1
\end{bmatrix},A^2_2 =
\begin{bmatrix}
1&1&0\\
0&0&1\\
1&0&1
\end{bmatrix}, A^3_2 =
\begin{bmatrix}
1&1&1\\
0&1&1\\
0&1&0
\end{bmatrix}
\end{align*}
and, for $\theta \in [-2,0]$, 	
\begin{align*}
B_1 (\theta)=
\begin{bmatrix}
2&0&0\\
0&1&0\\
2&0& \theta+2
\end{bmatrix},
B_2 (\theta)& =
\begin{bmatrix}
\theta+2&0&1\\
1&1&0\\
\theta+2&0&0
\end{bmatrix},
\end{align*}
are subject to unstructured perturbations of the form \eqref{unstructpert}. Defining
\begin{align*}
A^0_0=
\begin{bmatrix}
	-18&1&0\\
	1&-14&1\\
	1&1&-13
	\end{bmatrix},
A^1_0=
\begin{bmatrix}
	1&1&1\\
	1&0&1\\
	1&0&1
	\end{bmatrix},
A^2_0 =
\begin{bmatrix}
	1&1&0\\
	1&0&1\\
	1&1&2
	\end{bmatrix}, A^3_0 =
\begin{bmatrix}
1&1&1\\
1&1&1\\
0&1&0
\end{bmatrix},B_0(\theta)=\begin{bmatrix}
2&0&1\\
1&1&0\\
2&0&\theta+2
\end{bmatrix},\text{for}\; \theta\in [-2,0],
\end{align*}

\noindent we have, clearly,  $0\leq A^i_k\leq A^i_0,\  k=0, 1,2, \ i=1,2,3 $ and $0\leq B_k(\theta)\leq B_0(\theta),\  k=1,2, \forall \theta \in [-2,0]$. Moreover, it is easy to verify that the positive linear system of the form \eqref{triple} (with these matrices $A^i_0, B_0(\theta), i=0,1,2,3$)  is exponentially stable. Therefore, by Corollary \ref{theorem4.5} we get the following bound for the real unstructured stability radius of the time-delay switched linear system under consideration: 
$$
0.6008=\dfrac{1}{\|H_0^{-1}\|}\leq r_{\R}({\mathcal A},\Gamma) \leq \dfrac{1}{\max_{k=1,2}\|H_k^{-1}\|}=3.1875.$$			

\noindent Then, similarly as Example 1, it is easy to show, by numerical simulation, that the perturbed switched system of the form 
	\begin{align}
	\label{SwDDS-1}
	\dot{x}(t)=\widetilde A^0_{\sigma(t)}x(t)+\sum_{i=1}^{3}\widetilde A^i_{\sigma(t)}x(t-2)+\int_{-2}^0\widetilde B_{\sigma(t)}(\theta)x(t+\theta)d\theta, t\geq 0, \sigma\in \Sigma_+,
\end{align}
\noindent (with system's matrices being defined by \eqref{unstructpert}) is  exponentially stable under arbitrary switching $\sigma \in \Sigma_+$  and for any perturbations satisfying $\max_{k=1,2}\big\{ \sum_{i=0}^3\|\Delta^i_k\|+ \int_{-2}^0\|C_k(\theta)\|d\theta\big\}<0.6008$.
\section{Concluding Remarks}
We have presented a unified approach to study the robustness of exponential stability of time-delay  switched systems, under arbitrary switching,  described by linear functional differential equations, by introducing  the notion of  stability radius w.r.t. structured affine perturbations of the subsystem's matrices. As the main contribution,  we obtain some formulas to estimate this radius for some classes of time-delay switched linear systems, under the assumption on the existence of a common linear copositive Lyapunov function (LCLF) of the upper bounding linear positive subsystems. In the case of positive switched linear systems with discrete multiple time delays and/or distributed time delay, our general results  yield easily computable formulas to estimate the system's stability radius. To the best of our knowledge, such kind of results for switched systems have been obtained for the first time in this paper. 
It is our belief that the  approach developed in this paper is applicable to deal with similar problems but under less restrictive assumptions (e.g. when the nominal time-delay subsystems have  a common quadratic Lyapunov - Krasovskii functional and are subjected to more general types of parameter perturbations) which would hopefully bring about better and less conservative estimates for the system's stability radius. The other possibility of improving the results, as mentioned in Remark \ref{conservatism}, is to relax the assumption \eqref{cond2} on the existence of common LCLF, by using, instead of $\Sigma_+$, the more restrictive class of switching signals with average dwell time (ADT), as done in \cite{geromel, Qi, Dong2015} and in our recent work \cite{SonNgocNAHS}. These problems are the topics of our future works.

\section*{Acknowledgments}
\noindent  This work was supported partly by VAST (Vietnam Academy of Science and Technology) through the project QTRU03.02 /18-19. The paper is completed when  the first author spent a research stay at the Vietnam Institute for Advanced Studies in Mathematics (VIASM).












\end{document}